\pdfoutput=1
\documentclass[12pt]{article}
\usepackage{amsmath,amsfonts,amssymb,amsthm}
\usepackage{pdfpages}
\newcommand{\N}{\mathbb{N}}
\newcommand{\Z}{\mathbb{Z}}

\newcommand{\C}{\mathbb{C}}

\newcommand{\twosum}[2]{\sum_{\substack{#1\\#2}}}

\newtheorem{thm}{Theorem}[section]
\newtheorem{lem}[thm]{Lemma}
\begin{document}
\title{Estimates for character sums and Dirichlet $L$-functions to smooth moduli}
\author{A.J. Irving\\
Centre de recherches math\'ematiques, Universit\'e de Montr\'eal}
\date{}

\maketitle

\begin{abstract}
We use the $q$-analogue of van der Corput's method to estimate short character sums to smooth moduli.  If $\chi$ is a primitive Dirichlet character modulo a squarefree, $q^\delta$-smooth integer $q$ we show that 
$$L(\frac12,\chi)\ll_\epsilon q^{\frac{27}{164}+O(\delta)+\epsilon}.$$
\end{abstract}

\section{Introduction}

Suppose $\chi$ is a nontrivial Dirichlet character to modulus $q$.  An important question in analytic number theory is  the estimation  of the character sum 
$$S=\sum_{M<n\leq M+N}\chi(n).$$
Orthogonality of characters gives 
$$\sum_{n\pmod q}\chi(n)=0$$
so it is enough to consider the case that the length, $N$, of the sum satisfies $N<q$.  It is then conjectured that the sum should exhibit squareroot cancellation.  Specifically, for any $\epsilon>0$ we expect that
\begin{equation}\label{conj}
S\ll_\epsilon q^\epsilon N^{\frac12}.
\end{equation}
This is far from being proven but there are various results which give partial progress on the problem for certain ranges of $N$.

A number of applications of estimates for $S$ involve Dirichlet $L$-functions.  For example, when proving bounds for $L(s,\chi)$ the dependence on $q$ is determined by the quality of available results on $S$.  This then has implications for the zero-free regions and zero density results that one can prove.  If $\Re s$ is close to $1$ then it is useful to study very short sums for which $\log N=o(\log q)$.  On the other hand, if one wishes to estimate $L(\frac12,\chi)$ then the case $N\approx \sqrt{q}$ is important.  We will consider the latter situation in this work.  

The first nontrivial estimate for $S$ was found in 1918  by P\'olya and Vinogradov who showed that 
$$S\ll \sqrt{q}\log q.$$
For a proof see, for example, Iwaniec and Kowalski \cite[Theorem 12.5]{ik}.  This establishes (\ref{conj}) when the sum is long, $N\gg q$.  For shorter sums it gives a nontrivial estimate provided that $N\gg \sqrt{q}\log q$.  

Burgess \cite{burgess57, burgess62, burgess63} devised a method for estimating the sum $S$ when it is too short for the P\'olya-Vinogradov inequality to be applied.  His result gives 
$$S\ll_{\epsilon,r} N^{1-\frac{1}{r}}q^{\frac{r+1}{4r^2}+\epsilon},$$
for any $\epsilon>0$ and $r\in\N$, provided that either $q$ is cube-free or $r\leq 3$.  In the former situation it gives a power-saving for $S$ provided that $N\geq q^{\frac14+\delta}$, for some $\delta>0$.  Taking $r=2$ and $N\ll\sqrt{q}$ the Burgess estimate becomes 
$$S\ll_\epsilon q^{\frac{7}{16}+\epsilon}.$$
Burgess deduced from this that 
$$L(\frac12,\chi)\ll_\epsilon q^{\frac{3}{16}+\epsilon}.$$
This improves on the trivial ``convexity'' bound of $q^{\frac14}$.

Burgess's result is currently the sharpest known for arbitrary moduli, in particular it has not been improved for prime $q$.  However, by using methods based on Weyl differencing, various improvements can be given if we make assumptions on the prime factorisation of $q$.  In particular better results are known if $q$ is either powerful or smooth.

We say that $q$ is powerful if 
$$\prod_{p|q}p$$
is small relative to $q$.  In that case the sum may be estimated by a result of Iwaniec \cite{iwaniecdirichlet}, which extends previous work on prime-power moduli of Postnikov \cite{postnikov} and Gallagher \cite{gallagher}.  These works contain various applications to $L(s,\chi)$ when $\Re s$ is close to $1$.  In a recent paper \cite{milicevic}, Mili\'cevi\'c developed a quite general form of $p$-adic van der Corput method.  When $q$ is a prime power this can be used to give estimates for $S$ which are completely analogous to those arising from the classical theory of exponent pairs.  If $q$ is a prime power, $q=p^n$, then Mili\'cevi\'c obtained the estimate 
$$L(\frac12,\chi)\ll_\epsilon p^rq^{\theta+\epsilon},$$
for some fixed $r$ and $\theta\approx 0.1645$.  

This paper is concerned with smooth moduli $q$, for which the first result was given by Heath-Brown in \cite{rhbhybrid}.  He showed that if $q=q_0q_1$ is squarefree then 
\begin{equation}\label{rhbbound}
S\ll_\epsilon q^\epsilon\left(\sqrt{N}q_1^{\frac12}+\sqrt{N}q_0^{\frac14}\right).
\end{equation}
We say that $q$ is $q^\delta$-smooth if all the prime factors of $q$ are at most $q^\delta$.  In that case we can find a factorisation $q=q_0q_1$ with $q_1\in [q^{\frac13},q_1^{\frac13+\delta}]$ and (\ref{rhbbound}) becomes 
\begin{equation}\label{rhbbound2}
S\ll_\epsilon \sqrt{N}q^{\frac16+O(\delta)+\epsilon}.
\end{equation}
If $N\ll \sqrt{q}$ then this gives 
$$S\ll_\epsilon q^{\frac{5}{12}+O(\delta)+\epsilon}$$
which is better than the Burgess bound for sufficiently small $\delta$.  Heath-Brown deduced from this that 
$$L(\frac12,\chi)\ll_\epsilon q^{\frac16+O(\delta)+\epsilon}.$$
This is analogous to Weyl's estimate 
$$\zeta(\frac12+it)\ll_\epsilon t^{\frac16+\epsilon}.$$

The key idea in proving the estimate (\ref{rhbbound}) is to apply  van der Corput differencing  (called the $A$-process) to reduce the modulus of the sum from $q$ to $q_0$.  This latter sum may then be completed, which is an analogue of the van der Corput $B$-process, to give the result.  Graham and Ringrose \cite{grahamringrose} generalised this procedure, giving a $q$-analogue of the van der Corput $A^kB$ estimate.  A consequence of their result is that for any $\eta>0$ there exists a $\delta>0$ such that if $q$ is $q^\delta$ smooth and $N\geq q^\eta$ then one can give a power-saving for $S$.  This has applications to the distribution of zeros near $\Re s=1$ of $L(s,\chi)$, but not to $L(\frac12,\chi)$ since the optimal choice of $k$ is then $1$.  

By combining the methods for powerful and smooth $q$ one can obtain results for a larger set of moduli.  This was recently studied by Chang \cite{chang}, in which the object is to obtain a nontrivial bound for $N$ as small as possible.  On the other hand Goldmakher \cite{goldmakher} studied long sums, that is $N$ almost as large as $q$.  He demonstrated that for certain $q$ a small improvement on the P\'olya-Vinogradov inequality is possible.

The aim of this paper is to prove an estimate for $S$ which, when $q$ is sufficiently smooth,  improves on (\ref{rhbbound2}) when $N\approx \sqrt{q}$.  For such $q$ we will deduce an estimate for $L(\frac12,\chi)$ which breaks the Weyl barrier of $q^{\frac16}$.   

\begin{thm}\label{charthm}
Suppose that $q$ is squarefree and $q^\delta$ smooth for some $\delta>0$.  Then, for any primitive Dirichlet character $\chi$ to modulus $q$, any $N\leq q$  and any $\epsilon>0$ we have 
$$S=\sum_{M\leq n\leq M+N}\chi(n)\ll_\epsilon N^{\frac{23}{41}}q^{\frac{11}{82}+O(\delta)+\epsilon}.$$
\end{thm}

\begin{thm}\label{Lthm}
For $q$ and $\chi$ as in Theorem \ref{charthm} and any $\epsilon>0$ we have 
$$L(\frac12,\chi)\ll_\epsilon q^{\frac{27}{164}+O(\delta)+\epsilon}.$$
\end{thm}

Theorem \ref{charthm} will be proved by means of a $q$-analogue of the van der Corput $ABA^3B$ process.  This is unsurprising since the original van der Corput $ABA^3B$ estimate can be used to obtain the bound 
$$\zeta(\frac12+it)\ll_\epsilon t^{\frac{27}{164}+\epsilon},$$
see Graham and Kolesnik \cite[Theorem 4.1]{gk} for a proof.  
The proof of Theorem \ref{charthm} has many similarities with our work \cite{mydivisor}, in which a $q$-analogue of $BA^kB$ was used to estimate short Kloosterman sums.  However, the addition of an initial $A$-process leads to some technical changes.  As in \cite{mydivisor} a crucial role is played by an estimate for a complete multidimensional exponential sum.  This takes a different form to that in \cite{mydivisor} due to the replacement of Kloosterman fractions by Dirichlet characters as well as the extra differencing step.  The key estimate is proved in an appendix by Fouvry, Kowalski and Michel, using tools developed in their work \cite{fkm}.  

The exponents in Theorem \ref{charthm}, and similar results, can be predicted by means of the theory of exponent pairs.  This is described, in the context of the classical van der Corput method, by Graham and Kolesnik in \cite[Chapter 3]{gk}.  Suppose that $(k,l)$ is an exponent pair derived from the trivial pair $(0,1)$ by applications of the $A$ and $B$ processes.  We then expect that by applying the $q$-analogue of the same sequence of processes we can establish the bound 
\begin{equation}\label{generalbound}
\sum_{M\leq n<M+N}f(n)\ll_{\epsilon,k,l} q^{O_{k,l}(\delta)+\epsilon}(q/N)^k N^l,
\end{equation}
for many commonly occurring functions $f:\Z/q\Z\rightarrow\C$, provided that $q$ is squarefree, $q^\delta$-smooth and larger than $N$.  For example, the result of Theorem \ref{charthm} may be written 
$$S\ll_\epsilon q^{O(\delta)+\epsilon}(q/N)^{\frac{11}{82}}N^{\frac{57}{82}}$$
and $ABA^3B(0,1)=(\frac{11}{82},\frac{57}{82})$.  Similarly, Heath-Brown's estimate (\ref{rhbbound2}) corresponds to the pair $AB(0,1)=(\frac16,\frac23)$.  
 
Our exponent $\frac{27}{164}\approx 0.1647$ in Theorem \ref{Lthm} is slightly larger than $0.1645$  obtained by Mili\'cevi\'c.  However, his result only applies to prime-power moduli, a density $0$ set of $q$, whereas a positive proportion of $q$ will be sufficiently smooth and therefore satisfy the hypotheses of Theorem \ref{Lthm}.   In order to obtain the same exponent as Mili\'cevi\'c we would need to iterate the $A$ and $B$ processes an arbitrary number of times, thereby obtaining a sufficiently general form of (\ref{generalbound}).  This would be technically challenging, but a more significant issue would be to give bounds for the progressively more complex complete sums which would arise.  

Throughout this work we adopt the standard convention that $\epsilon$ denotes a small positive quantity whose value may differ at each occurrence.  If $f$ is a function on $\Z/q\Z$ then we will, without comment, extend $f$ to a function on $\Z$ (or on $\Z/q'\Z$ for any multiple $q'$ of $q$).  

\subsection*{Acknowledgements}

This work was completed whilst I was a CRM-ISM postdoctoral fellow at the Universit\'e de Montr\'eal.    I am very grateful to professors Fouvry, Kowalski and Michel for their appendix proving the crucial complete sum estimate.  

\section{The $A$-Process}

In this section we will describe the $q$-analogue of the van der Corput $A^k$-process for quite general functions $f$.  We begin with a single iteration.  When $f$ is a Dirichlet character the following lemma was first proved by Heath-Brown in \cite{rhbhybrid}, it has since been modified for a number of different $f$.  A result with a similar level of generality to ours was given by Blomer and Mili\'cevi\'c in \cite[Lemma 12]{blomermil}.  

\begin{lem}\label{onediff}
Suppose $q=q_0q_1$.  For $i\in \{1,2\}$ let $f_i:\Z/q_i\Z\rightarrow\C$ with $f_i(n)\ll 1$ for all $n$.  Let $f(n)=f_0(n)f_1(n)$ be a function on $\Z/q\Z$  and let $I$ be an interval of length at most $N$.  We then have   
$$\left|\sum_{n\in I}f(n)\right|^2\ll q_1\left(N+\sum_{0<|h|\leq N/q_1}\left|\sum_{n\in I(h)}f_0(n)\overline{f_0(n+q_1h)}\right|\right),$$ 
where $I(h)$ are subintervals of $I$.  
\end{lem} 

\begin{proof}
If $N<q_1$ then $q_1N>N^2$ so the result follows from the trivial bound 
$$\sum_{n\in I}f(n)\ll N.$$
We therefore assume, for the remainder of the proof, that $N\geq q_1$.  We let 
$$a_n=\begin{cases}
f(n) & n\in I\\
0 & \text{otherwise.}\\
\end{cases}$$
The sum of interest is 
$$S=\sum_n a_n.$$
Since $N\geq q_1$ we have $H=[N/q_1]\geq 1$ and thus, using that $f_1$ has period $q_1$, we obtain
$$HS=\sum_{h=1}^H \sum_n a_{n+q_1h}=\sum_nf_1(n)\twosum{h=1}{n+q_1h\in I}f_0(n+q_1h).$$
The outer summation is supported on an interval of length $O(N)$ so we may use Cauchy's inequality and the bound $f_1(n)\ll 1$ to deduce that 
$$H^2S^2\leq N\sum_n\left|\twosum{h=1}{n+q_1h\in I}f_0(n+q_1h)\right|^2=N\sum_{h_1,h_2=1}^H \twosum{n}{n+q_1h_i\in I}f_0(n+q_1h_1)\overline{f_0(n+q_1h_2)}.$$
Changing variable in the final summation this becomes 
\begin{eqnarray*}
H^2S^2&\leq& N\sum_{h_1,h_2=1}^H \twosum{n\in I}{n+q_1(h_2-h_1)\in I}f_0(n)\overline{f_0(n+q_1(h_2-h_1))}\\
&\ll& NH\sum_{|h|\leq H}\left|\sum_{n\in I(h)}f_0(n)\overline{f_0(n+q_1h)}\right|,
\end{eqnarray*}
where $I(h)$ is a subinterval of $I$.  We apply the bound $f_0(n)\ll 1$ to the terms with $h=0$ and conclude that 
$$S^2 \ll NH^{-1}\left(N+\sum_{0<|h|\leq H}\left|\sum_{n\in I(h)}f_0(n)\overline{f_0(n+q_1h)}\right|\right).$$
The result follows since $NH^{-1} \ll q_1$.  
\end{proof}

For any complex-valued function $f$ we introduce the notation 
$$f(n;h_1,\ldots,h_k)=\prod_{I\subseteq \{1,\ldots,k\}}f\left(n+\sum_{i\in I}h_i\right)^{\sigma(I)},$$
where $\sigma(I)$ denotes that the complex conjugate is taken when $|I|$ is odd. The term $I=\emptyset$ is to be included in the product.  The result of the last lemma may therefore be written 
$$\left|\sum_{n\in I}f(n)\right|^2\ll q_1\left(N+\sum_{0<|h|\leq N/q_1}\left|\sum_{n\in I(h)}f_0(n;q_1h)\right|\right).$$ 
 Iterating this $k$ times we derive the following.

\begin{lem}\label{kdiff}
Suppose that $k\in \N$ and $q=q_0q_1\ldots q_k$.  For $0\leq i\leq k$ let $f_i:\Z/q_i\Z\rightarrow \C$ be functions with $f_i(n)\ll 1$ and let $f(n)=\prod_{i=0}^k f_i(n)$.   If $I$ is an interval of length at most $N$ then 
\begin{eqnarray*}
\lefteqn{\left|\sum_{n\in I}f(n)\right|^{2^k}}\\
&\ll_k& \sum_{j=1}^k N^{2^k-2^{k-j}}q_{k-j+1}^{2^{k-j}}\\
&&\hspace{1cm}+N^{2^k-k-1}(q/q_0)\sum_{0<|h_1|\leq N/q_1}\ldots\sum_{0<|h_k|\leq N/q_k}\left|\sum_{n\in I(h_1,\ldots,h_k)}f_0(n;q_1h_1,\ldots,q_kh_k)\right|.
\end{eqnarray*}
In the last line $I(h_1,\ldots,h_k)$ denotes a subinterval of $I$ and $f_0(n;q_1h_1,\ldots,q_kh_k)$ was defined above.
\end{lem}

\begin{proof}
The result can be proved by induction using Lemma \ref{onediff}.  The details, for specific choices of $f$, have previously been given in Heath-Brown \cite{rhbx32} and our paper \cite{mydivisor}.  The proof is unchanged in the current,  more general, setting so we do not include it.
\end{proof}

\section{The $B$-Process}

The $q$-analogue of the van der Corput $B$-process is more commonly known simply as completion of a sum.  It has been known, in some form,  at least since the work of P\'olya and Vinogradov.   In order to state it we define the Fourier transform $\hat f:\Z/q\Z\rightarrow \C$ of a function $f:\Z/q\Z\rightarrow \C$ by 
$$\hat f(x)=\frac{1}{\sqrt{q}}\sum_{n\pmod q}f(n)e_q(nx),$$
where 
$$e_q(y)=e^{\frac{2\pi iy}{q}}.$$

\begin{lem}\label{Bprocess}
Suppose we are given a function $f:\Z/q\Z\rightarrow \C$ and an interval $I\subseteq [M,M+N)$ with $N\leq q$.  We assume, without loss of generality, that $M\in\Z$.  We then have 
$$\sum_{n\in I}f(n) \ll \frac{N\hat f(0)}{\sqrt{q}}+\frac{N\log q}{\sqrt q}\max_J\left|\sum_{x\in J}\hat f(x)e_q(-Mx)\right|,$$
where the maximum is taken over all intervals $J$ of length at most $q/N$ which do not contain $0$.  
\end{lem}

\begin{proof}
The Plancherel identity gives 
$$S:=\sum_{n\in I}f(n)=\sum_{x\pmod q}\hat f(x)\overline{\hat I(x)},$$
where $\hat I$ is the Fourier transform of the indicator function of $I$.  We can write 
$$\hat I(x)=\frac{1}{\sqrt q}\sum_{n\in I}e_q(nx)=e_q(Mx)\frac{1}{\sqrt q}\twosum{n\leq N}{M+n\in I}e_q(nx)=e_q(Mx)g(x),$$ 
say.  Since $g$ is given by a geometric series we have the well-known estimate 
$$g(x)\ll \frac{1}{\sqrt{q}}\min(N,\frac{1}{\|x/q\|}),$$
in which $\|.\|$ denotes the distance from a real number to the nearest integer.    In simple applications of completion this is used in conjunction with an estimate for the complete sums $\hat f(x)$.  In this work however, as in \cite{mydivisor}, we wish to be able to detect cancellation between the $\hat f(x)$ by means of applications of the $A$-process.  We therefore split the sum over $x$ into short intervals, on which the factor $g(x)$ may be removed by partial summation.

We identify the summation range $x\pmod q$ with the integers in $(-\frac{q}{2},\frac{q}{2}]$ and extend $g(x)$ to a smooth  function of $x$ on that interval.  We therefore write 
$$S\ll \frac{N\hat f(0)}{\sqrt{q}}+\sum_{-q/2<x<0}\hat f(x)e_q(-Mx)\overline{g(x)}+\sum_{0<x\leq q/2}\hat f(x)e_q(-Mx)\overline{g(x)}.$$ 
For $x\ne 0$ the above estimate for $g(x)$  gives 
$$g(x)\ll \frac{\sqrt{q}}{x}.$$
In addition we have 
$$g'(x)=\frac{2\pi i}{\sqrt{q}} \twosum{n\leq N}{M+n\in I}\frac{n}{q}e_q(nx)\ll \frac{N}{\sqrt{q}x}.$$
This enables us to remove the weight $g(k)$ on intervals of length $q/N$. Specifically, we let $K=[q/N]$ and use partial summation to show that for any $L\leq K$ and any $r\in \N$ we have
$$\sum_{(r-1)K<x\leq (r-1)K+L}\hat f(x)e_q(-Mx)\overline{g(x)}\ll \frac{N}{\sqrt{q}r}\max_{0<L'\leq K}\left|\sum_{(r-1)K<x\leq (r-1)K+L'}\hat f(x)e_q(-Mx)\right|.$$
An analogous bound holds on intervals of length $K$ when $x<0$.  We may therefore divide the sum into $O(N)$ intervals of this form and conclude that 
$$S\ll \frac{N\hat f(0)}{\sqrt{q}}+\frac{N\log q}{\sqrt q}\max_J\left|\sum_{x\in J}\hat f(x)e_q(-Mx)\right|,$$
where the maximum is as described in the statement of the lemma.  The term $\log q$ arises from the estimate 
$$\sum_{r\ll N}\frac1r \ll \log N+1\ll \log q.$$
\end{proof}

We also give a version of the last lemma which is less precise but which allows the length of the sum to exceed the modulus $q$.  This is essentially the usual completion of sums, used by P\'olya and Vinogradov etc, so we do not include a proof.  An almost identical result is given by Iwaniec and Kowalski in \cite[Lemma 12.1]{ik}.  

\begin{lem}\label{completion}
Suppose we are given a function $f:\Z/q\Z\rightarrow \C$ and an interval $I\subseteq [M,M+N)$. We then have 
$$\sum_{n\in I}f(n) \ll \frac{N\hat f(0)}{\sqrt{q}}+\frac{1}{\sqrt{q}}\sum_{x\not\equiv 0\pmod q}\frac{|\hat f(x)|}{\|x/q\|}.$$
\end{lem}

\section{Complete Sums}

In the course of our proof of Theorem \ref{charthm} we will need to deal with two different complete  sums to a squarefree modulus.  In both cases the Chinese Remainder Theorem is used to write the sum as a product of complete sums to prime moduli.  The latter are then bounded by an appeal to a suitable form of the Riemann hypothesis over finite fields.  The first sum is $1$-dimensional so the necessary result follows from the work of Weil \cite{weil1}.  The second sum is much more complex.  The necessary result can be derived from the work of Deligne \cite{deligne2}, the details are given by Fouvry, Kowalski and Michel in the appendix.  

\subsection{A $1$-Dimensional Sum}

Suppose $\chi\pmod q$ is a primitive Dirichlet character.  For integers $h$ and $x$ we consider the sum 
$$W_{\chi,h}(x)=\sum_{n\pmod q}\chi(n)\overline{\chi(n+h)}e_q(nx).$$
A more general form of this sum was discussed by Graham and Ringrose in \cite[Section 4]{grahamringrose}.  We begin by observing that if $(a,q)=1$ then 
\begin{eqnarray*}
W_{\chi,h}(ax)&=&\sum_{n\pmod q}\chi(n)\overline{\chi(n+h)}e_q(anx)\\
&=&\sum_{n\pmod q}\chi(\overline an)\overline{\chi(\overline an+h)}e_q(nx)\\
&=&\sum_{n\pmod q}\chi(n)\overline{\chi(n+ah)}e_q(nx)\\
&=&W_{\chi,ah}(x).\\
\end{eqnarray*}
Combining this with a special case of \cite[Lemma 4.1]{grahamringrose} we can deduce the following multiplicative property.  

\begin{lem}\label{Wmult}
Suppose $q=uv$ with $(u,v)=1$ and that $\chi$ is a primitive Dirichlet character modulo $q$.  Then there exist primitive characters $\chi_u\pmod u$ and $\chi_v\pmod v$ such that 
$$W_{\chi,h}(x)=W_{\chi_u,\overline vh}(x)W_{\chi_v,\overline uh}(x)$$
where $u\overline u\equiv 1\pmod v$ and $v\overline v\equiv 1\pmod u$.
\end{lem}

Using the results of Weil \cite{weil1}, Graham and Ringrose \cite[Lemmas 4.2 and 4.3]{grahamringrose} obtained the following estimate for the sum to a prime modulus.  Our bound is slightly weaker than theirs as we do not include their better estimate when $p|h$ but $p\nmid x$.  

\begin{lem}\label{weilprime}
If $p$ is prime and $\chi$ is a primitive Dirichlet character modulo $p$ then for any integers $h$ and $x$ we have 
$$W_{\chi,h}(x)\ll \begin{cases}
p & h\equiv x\equiv 0\pmod p\\
\sqrt{p} & \text{otherwise.}\\
\end{cases}$$
\end{lem}

Observe that the bound in the lemma can be expressed more compactly as $p^{\frac12}(h,x,p)^{\frac12}$.  We may therefore combine our results to obtain a bound for squarefree moduli.

\begin{lem}\label{weilW}
Suppose $q$ is squarefree and $\chi\pmod q$ is a primitive character.  Then, for any integers $h,x$ and any $\epsilon>0$ we have 
$$W_{\chi,h}(x)\ll_\epsilon q^{\frac12+\epsilon}(h,x,q)^{\frac12}.$$
\end{lem}

\subsection{A Higher Dimensional Sum}

We now turn our  attention to a complete sum constructed from products of the sums $W$.  Recall that for a function $f(n)$ and integers $h_1,\ldots,h_k$ we defined 
$$f(n;h_1,\ldots,h_k)=\prod_{I\subseteq \{1,\ldots,k\}}f\left(n+\sum_{i\in I}h_i\right)^{\sigma(I)}.$$
Using this notation we let 
$$K_{\chi,h}(h_1,\ldots,h_k,y)=\sum_{x\pmod q}e_q(xy)W_{\chi,h}(x;h_1,\ldots,h_k).$$
We begin by establishing a multiplicative property.

\begin{lem}\label{Kmult}
If $\chi$ is a primitive character mod $q=uv$ with $(u,v)=1$ then 
$$K_{\chi,h}(h_1,\ldots,h_k,y)=K_{\chi_u,\overline vh}(h_1,\ldots,h_k,\overline vy)K_{\chi_v,\overline uh}(h_1,\ldots,h_k,\overline uy).$$
\end{lem}

\begin{proof}
By Lemma \ref{Wmult} we obtain 
$$W_{\chi,h}(x;h_1,\ldots,h_k)=W_{\chi_u,\overline vh}(x;h_1,\ldots,h_k)W_{\chi_v,\overline uh}(x;h_1,\ldots,h_k).$$
The result then follows from the Chinese Remainder Theorem.
\end{proof}

It remains to estimate the sums $K$ when the modulus is prime.  

\begin{lem}
Let $\chi$ be a primitive Dirichlet character modulo a prime $p$, suppose $h\not\equiv 0\pmod p$ and let $h_1,\ldots,h_k,y$ be integers.  
\begin{enumerate}
\item If  $y\not\equiv 0\pmod p$ or if 
$$\prod_{i=1}^k h_i\not\equiv 0\pmod p$$
then 
$$K_{\chi,h}(h_1,\ldots,h_k,y)\ll_k p^{\frac{2^k+1}{2}}.$$

\item Otherwise we have the trivial estimate
$$K_{\chi,h}(h_1,\ldots,h_k,y)\ll_k p^{\frac{2^k+2}{2}}.$$
\end{enumerate}
\end{lem}

\begin{proof}
By increasing the value of the implied constant in the result we may assume that $p>5$.   Let $S(y)$ be the sum defined in the appendix with 
$$n_1=n_2=2^{k-1}$$
and with shifts $t_i,s_j$ given by $\sum_{i\in I}h_i$ ($t_i$ corresponding to $|I|$ even and $s_j$ to $|I|$ odd).  We then have 
$$K_{\chi,h}(h_1,\ldots,h_k,y)=p^{2^{k-1}}S(y)$$ 
so Theorem 1 from the appendix may be applied.  Since $n_1=n_2$ we obtain $S(y)\ll_k \sqrt{p}$ provided that either $y\not\equiv 0\pmod p$ or that one of the shifts $\sum_{i\in I}h_i$ occurs an odd number of times modulo $p$.  However, we showed in \cite[Lemma 4.5]{mydivisor} that if the latter condition  fails then $h_i\equiv 0\pmod p$ for at least one index $i$ and thus 
$$\prod_{i=1}^k h_i\equiv 0\pmod p.$$
This completes the proof of the first estimate and the second follows directly from Lemma \ref{weilprime}.
\end{proof}

Combining the last two lemmas we deduce a bound for $K$ when $q$ is squarefree.

\begin{lem}\label{Kbound}
Suppose $\chi$ is a primitive Dirichlet character modulo the squarefree integer $q$, $(h,q)=1$ and let $h_1,\ldots,h_k,y$ be integers.  Then, for any $\epsilon>0$, we have 
$$K_{\chi,h}(h_1,\ldots,h_k,y)\ll_{\epsilon,k} q^{\frac{2^k+1}{2}+\epsilon}\left(q,y,\prod_{i=1}^k h_i\right)^{\frac12}.$$
\end{lem}

\section{Theorem \ref{charthm}: First Steps}

Before beginning the proof of Theorem \ref{charthm} we give the following well-known lemma.

\begin{lem}\label{gcdlem}
Suppose $q$ is an integer and $H\geq 1$.  For any $\epsilon>0$ we have 
$$\sum_{0<h\leq H}(h,q)\ll_\epsilon Hq^\epsilon$$
and 
$$\sum_{0<h\leq H}\frac{(h,q)}{h}\ll_\epsilon (Hq)^\epsilon.$$
\end{lem}

\begin{proof}
We have 
\begin{eqnarray*}
\sum_{0<h\leq H}(h,q)&=&\sum_{d|q}\twosum{0<h\leq H}{(h,q)=d}d\\
&\leq&\sum_{d|q}H\\
&\leq&H\tau(q)\ll_\epsilon Hq^\epsilon.\\
\end{eqnarray*}
The second claim can be proved by a similar argument.  
\end{proof}

We suppose that $q$ is squarefree, $\chi$ is a primitive Dirichlet character modulo $q$  and that $I$ is an interval of length at most $N<q$.  The sum of interest is 
$$S=\sum_{n\in I}\chi(n)$$
and we wish to establish that 
$$S\ll_\epsilon N^{\frac{23}{41}}q^{\frac{11}{82}+O(\delta)+\epsilon}.$$
This follows from the trivial bound if $N\leq q^{\frac{11}{36}}$ and it is weaker than the P\'olya-Vinogradov inequality if $N\geq q^{\frac{15}{23}}$.  We therefore assume that 
\begin{equation}\label{Nsize}
q^{\frac{11}{36}}\leq N\leq q^{\frac{15}{23}}.
\end{equation}
The assumption that $q$ is squarefree implies that whenever $de|q$ we have $(d,e)=1$.  We will use this repeatedly in what follows.  

If $q$ factorises as $q=q_0q_1$ then $(q_0,q_1)=1$.  It follows by the Chinese Remainder Theorem that $\chi=\chi_0\chi_1$ for primitive characters $\chi_i\pmod {q_i}$.  We may therefore apply the $A$-process once, Lemma \ref{onediff}, with this factorisation to deduce that 
$$S^2\ll q_1\left(N+\sum_{0<|h|\leq N/q_1}\left|S(q_1h)\right|\right)$$ 
with 
$$S(q_1h)=\sum_{n\in I(h)}\chi_0(n)\overline{\chi_0(n+q_1h)}.$$
The magnitude of the sum $S(q_1h)$ depends on $(q_1h,q_0)=(h,q_0)$.  We will therefore use standard techniques, such as M\"obius inversion, to reduce $S(q_1h)$ to sums, $T$, of the same form in which the modulus and shift are coprime.  As a result we will need estimates for sums whose length and modulus are  different from those in $S(q_1h)$.  The following gives the necessary bound.  It is important to note that the quantities $q$ and $N$ need not be the same as those occurring in Theorem \ref{charthm}.

\begin{lem}\label{BAkB}
Suppose that $q=q_0q_1\ldots q_k$ is squarefree, $\chi$ is a primitive Dirichlet character modulo $q$  and $(h,q)=1$.  Let $I$ be an interval of length  at most $N\leq q$ and consider the sum 
$$T=\sum_{n\in I}\chi(n)\overline{\chi(n+h)}.$$
For any $\epsilon>0$ we have 
$$T\ll_{\epsilon,k} q^\epsilon\left(\sum_{j=1}^kN^{2^{-j}}q^{1/2-2^{-j}}q_{k-j+1}^{2^{-j}}+N^{2^{-k}}q^{1/2-2^{-k}}q_0^{1/2^{k+1}}\right).$$
\end{lem}

The last lemma should be compared to our estimate \cite[Theorem 1.3]{mydivisor}, in which a similar bound was obtained for short Kloosterman sums.  Both results use a $q$-analogue of the van der Corput $BA^kB$ estimate and the proofs are essentially the same apart from the use of different estimates for complete exponential sums.  In \cite{mydivisor} the result contains an extra term $q^{1/2}q_0^{-1/2^{k+1}}$.  In the course of the proof of Lemma \ref{BAkB} we will show that this term is not needed; the same argument could be used to remove it from \cite[Theorem 1.3]{mydivisor}. In the remainder of this section we show that Lemma \ref{BAkB} implies Theorem \ref{charthm}, we will prove the former in the next section.

The next result specialises Lemma \ref{BAkB} to $k=3$ and smooth moduli $q$.

\begin{lem}\label{BA3B}
Suppose $q$ is squarefree and $q^\delta$-smooth for some $\delta>0$.  Let $h$, $\chi$, $I$, $N$ and $T$ be as in Lemma \ref{BAkB}.  Assuming that $\delta$ is sufficiently small we have 
$$T\ll_{\epsilon}N^{\frac16}q^{\frac{11}{30}+O(\delta)+\epsilon}.$$
\end{lem}

\begin{proof}
If $N\leq q^{\frac{11}{25}}$ then  $N^{\frac16}q^{\frac{11}{30}}\geq N$ so the result follows from the trivial bound $|T|\leq N$.  For the remainder of the proof we assume that $N\geq q^{\frac{11}{25}}$.  We apply Lemma \ref{BAkB} with $k=3$ and a factorisation $q=q_0q_1q_2q_3$ for which $q_j\approx Q_j$, where 
$$Q_0=q^{-\frac{2}{15}}N^{\frac23},$$
$$Q_1=q^{-\frac{1}{15}}N^{\frac13},$$
$$Q_2=q^{\frac{7}{15}}N^{-\frac13}$$
and
$$Q_3=q^{\frac{11}{15}}N^{-\frac23}.$$
Observe that $Q_0Q_1Q_2Q_3=q$.  Furthermore, by our assumptions that $q^{\frac{11}{25}}\leq N\leq q$, we know that all the $Q_j$ exceed some fixed power of $q$.    It follows that if $q$ is $q^\delta$-smooth, with $\delta$ sufficiently small, then we can find a factorisation with 
$$q_j\in [Q_jq^{-\delta},Q_j]\text{ for }j=1,2,3$$
and therefore 
$$q_0\in [Q_0,Q_0q^{3\delta}].$$
Note that this choice is not quite optimal if we wish to get the best dependence on $\delta$.  Lemma \ref{BAkB} now gives the required result.
\end{proof}

Observe that the estimate of Lemma \ref{BA3B} may be written 
$$T\ll_{\epsilon}q^{O(\delta)+\epsilon}(q/N)^{\frac{11}{30}}N^{\frac{8}{15}}$$
and that $(\frac{11}{30},\frac{8}{15})=BA^3B(0,1)$.  The sum $T$ may also be estimated by a direct application of Lemmas \ref{completion} and \ref{weilW}.  This gives the following result, which can handle the case $N>q$.

\begin{lem}\label{weilT}
Suppose $q$ is squarefree, $(h,q)=1$ and $\chi\pmod q$ is a primitive Dirichlet character modulo $q$.  Let $I$ be an interval of length at most $N$ and $T$ as in the previous lemmas.  We then have 
$$T\ll_\epsilon \left(\frac{N}{q}+1\right)q^{\frac12+\epsilon}.$$
\end{lem}

We now generalise to the case $(h,q)>1$.  We do not give an optimal treatment, preferring to give a simple estimate which is sufficiently sharp to prove Theorem \ref{charthm}.

\begin{lem}\label{Tbound}
Suppose that $\chi$ is a primitive Dirichlet character modulo the squarefree, $q^\delta$-smooth  integer $q$ and that $I$ is an interval of length at most $N\leq q$.  For any integer $h$ and any $\epsilon>0$ we have 
$$T\ll_{\epsilon}\sqrt{(h,q)}N^{\frac16}q^{\frac{11}{30}+O(\delta)+\epsilon}.$$
\end{lem}

\begin{proof}
We write $(h,q)=d$ and $q=de$ (recall $(d,e)=1$ since $q$ is squarefree).  Factorising $\chi=\chi_d\chi_e$ for primitive characters $\chi_d\pmod d$ and $\chi_e\pmod e$ we obtain 
\begin{eqnarray*}
T&=&\sum_{n\in I}\chi(n)\overline{\chi(n+h)}\\
&=&\sum_{n\in I}\chi_d(n)\chi_e(n)\overline{\chi_d(n+h)\chi_e(n+h)}\\
&=&\sum_{n\in I}|\chi_d(n)|^2\chi_e(n)\overline{\chi_e(n+h)}\\
&=&\twosum{n\in I}{(n,d)=1}\chi_e(n)\overline{\chi_e(n+h)}\\
 &=&\sum_{d'|d}\mu(d')\twosum{n\in I}{d'|n}\chi_e(n)\overline{\chi_e(n+h)}\\
&\leq&\sum_{d'|d}\left|\sum_{n\in I/d'}\chi_e(n)\overline{\chi_e(n+\overline{d'}h)}\right|.\\
\end{eqnarray*}
In the final line 
$$I/d'=\{x:d'x\in I\}$$
and 
$$\overline{d'}d'\equiv 1\pmod e.$$
We have two bounds for the final sum, which has length at most $N/d'\leq N$ and modulus $e$.  Firstly, Lemma \ref{weilT} gives 
$$\sum_{n\in I/d'}\chi_e(n)\overline{\chi_e(n+\overline{d'}h)}\ll_\epsilon \left(\frac{N}{e}+1\right)e^{\frac12+\epsilon}.$$
Secondly, we can use Lemma \ref{BA3B}.  Since $e$ divides $q$ we know that it is $q^\delta$ smooth.  Therefore, if $e\geq q^{O(1)}$ then it is $e^{O(\delta)}$ smooth.  Therefore, provided that $N\leq e$, we obtain a bound 
$$N^{\frac16}e^{\frac{11}{30}+O(\delta)+\epsilon}\leq N^{\frac16}q^{\frac{11}{30}+O(\delta)+\epsilon}.$$
Combining these two estimates we obtain 
$$\sum_{d'|d}\left|\sum_{n\in I/d'}\chi_e(n)\overline{\chi_e(n+\overline{d'}h)}\right|\ll_\epsilon N^{\frac16}q^{\frac{11}{30}+O(\delta)+\epsilon}+\frac{N}{\sqrt{e}}$$
and thus 
$$T\ll_\epsilon N^{\frac16}q^{\frac{11}{30}+O(\delta)+\epsilon}+\frac{N\sqrt{(h,q)}}{\sqrt{q}}\leq \sqrt{(h,q)}\left(N^{\frac16}q^{\frac{11}{30}+O(\delta)+\epsilon}+\frac{N}{\sqrt{q}}\right).$$
The first term is largest since we are assuming that $N\leq q$ and thus the result follows.
\end{proof}

We now return to the sum $S$.  Recall that we have shown that 
$$S^2\ll q_1\left(N+\sum_{0<|h|\leq N/q_1}\left|S(q_1h)\right|\right)$$ 
where
$$S(q_1h)=\sum_{n\in I(h)}\chi_0(n)\overline{\chi_0(n+q_1h)}.$$
The quantity $q_0$ is $q^\delta$ smooth and therefore $q_0^{O(\delta)}$-smooth provided that $q_0\geq q^{O(1)}$.  Assuming $N\leq q_0$ we use Lemma \ref{Tbound} to obtain 
\begin{eqnarray*}
S^2&\ll_\epsilon& q^{O(\delta)+\epsilon} q_1\left(N+N^{\frac16}q_0^{\frac{11}{30}}\sum_{0<|h|\leq N/q_1}\sqrt{(h,q_0)}\right)\\
&\ll_\epsilon&q^{O(\delta)+\epsilon} \left(Nq_1+N^{\frac76}q_0^{\frac{11}{30}}\right),\\
\end{eqnarray*}
the final estimate following from Lemma \ref{gcdlem}.

We now let 
$$Q_0=N^{-\frac{5}{41}}q^{\frac{30}{41}}$$
and
$$Q_1=N^{\frac{5}{41}}q^{\frac{11}{41}}.$$
Observe that $Q_0Q_1=q$ and that, for $N\leq q$,  $Q_0,Q_1\geq q^{O(1)}$.  Therefore, if $\delta$ is sufficiently small, we may factorise $q=q_0q_1$ with 
$$q_0\in [Q_0,Q_0q^\delta]$$
and
$$q_1\in  [Q_1q^{-\delta},Q_1].$$
By our assumption (\ref{Nsize}) on the size of $N$ we know that $N\leq q_0$ and therefore the above arguments give 
$$S^2\ll_\epsilon N^{\frac{46}{41}}q^{\frac{11}{41}+O(\delta)+\epsilon}.$$
This completes the proof of Theorem \ref{charthm}.

\section{Proof of Lemma \ref{BAkB}}

We begin by proving the result in the case $\frac{q}{N}\leq q_0$.  Suppose that $I\subseteq [M,M+N]$.  We apply the $B$-process, Lemma \ref{Bprocess}, with $f(n)=\chi(n)\overline{\chi(n+h)}$, to obtain 
$$T\ll \frac{N\hat f(0)}{\sqrt{q}}+\frac{N\log q}{\sqrt q}\max_J\left|\sum_{x\in J}\hat f(x)e_q(-Mx)\right|,$$
where the maximum is taken over all intervals $J$ of length at most $q/N$ which do not contain $0$.  The Fourier transform is given by 
$$\hat f(x)=\frac{1}{\sqrt{q}}\sum_{n\pmod q}\chi(n)\overline{\chi(n+h)}e_q(nx)=\frac{1}{\sqrt{q}}W_{\chi,h}(x).$$
We know that $(h,q)=1$ and therefore Lemma \ref{weilW} gives 
$$\hat f(x)\ll_\epsilon q^\epsilon.$$
We may thus deduce that 
$$\frac{N\hat f(0)}{\sqrt{q}}\ll_\epsilon Nq^{-\frac12+\epsilon}$$
which is certainly sufficiently small.  Lemma \ref{BAkB} will therefore follow if we can show that 
\begin{eqnarray*}
\lefteqn{\sum_{x\in J}\frac{1}{\sqrt q}W_{\chi,h}(x)e_q(-Mx) }\\
&\ll_{\epsilon,k}&\frac{q^{\frac12+\epsilon}}{N}\left(\sum_{j=1}^kN^{2^{-j}}q^{1/2-2^{-j}}q_{k-j+1}^{2^{-j}}+N^{2^{-k}}q^{1/2-2^{-k}}q_0^{1/2^{k+1}}\right)\\
\end{eqnarray*}
for all intervals $J$ of length at most $q/N$.  On writing $K=q/N$ we see that it is sufficient to establish 
\begin{eqnarray*}
\lefteqn{\sum_{x\in J}\frac{1}{\sqrt q}W_{\chi,h}(x)e_q(-Mx) }\\
&\ll_{\epsilon,k}&q^{\epsilon}\left(\sum_{j=1}^kK^{1-2^{-j}}q_{k-j+1}^{2^{-j}}+K^{1-2^{-k}}q_0^{1/2^{k+1}}\right).\\
\end{eqnarray*}
This final estimate should be compared with existing $A^kB$ results, for example that given by Heath-Brown in \cite[Theorem 2]{rhbx32}.  In order to prove it we begin by applying Lemma \ref{kdiff} with 
$$f(x)=\frac{1}{\sqrt q}W_{\chi,h}(x)e_q(-Mx).$$
By Lemma \ref{Wmult} and the Chinese Remainder Theorem we may factorise this as 
$$f(x)=\prod_{i=0}^k f_i(x)$$ 
where
$$f_i(x)=\frac{1}{\sqrt{q_i}}W_{\chi_i,\overline{q/q_i}h}(x)e_{q_i}(-M\overline{q/q_i}x),$$
for some primitive characters $\chi_i\pmod{q_i}$.  We compute 
\begin{eqnarray*}
f_0(x;h_1,\ldots,h_k)&=&\prod_{I\subseteq \{1,\ldots,k\}}f_0\left(x+\sum_{i\in I}h_i\right)^{\sigma(I)}\\
&=&q_0^{-2^{k-1}}W_{\chi_0,\overline{q/q_0}h}(x;h_1,\ldots,h_k)\prod_{I\subseteq \{1,\ldots,k\}}e_{q_0}(-M\overline{q/q_0}(x+\sum_{i\in I}h_i)^{\sigma(I)}.\\
\end{eqnarray*}
Since $I$ runs over the same number of even and odd subsets we have
$$\prod_{I\subseteq \{1,\ldots,k\}}e_{q_0}(-Mx)^{\sigma(I)}=1.$$
It follows that 
$$\prod_{I\subseteq \{1,\ldots,k\}}e_{q_0}(-M\overline{q/q_0}(x+\sum_{i\in I}h_i))^{\sigma(I)}$$
depends only on the $h_i$, not on $x$, and therefore it can be taken outside the sum in Lemma \ref{kdiff}.  We obtain 
\begin{eqnarray*}
\lefteqn{\left|\sum_{x\in J}\frac{1}{\sqrt q}W_{\chi,h}(x)e_q(-Mx) \right|^{2^k}}\\
&\ll_{\epsilon,k}& q^\epsilon\left(\sum_{j=1}^k K^{2^k-2^{k-j}}q_{k-j+1}^{2^{k-j}}\right.\\
&&\hspace{1cm}\left.+K^{2^k-k-1}(q/q_0)\sum_{0<|h_1|\leq K/q_1}\ldots\sum_{0<|h_k|\leq K/q_k}\right.\\
&&\hspace{1cm}\left.q_0^{-2^{k-1}}\left|\sum_{x\in J(h_1,\ldots,h_k)}W_{\chi_0,\overline{q/q_0}h}(x;q_1h_1,\ldots,q_kh_k)\right|\right).\\
\end{eqnarray*}
We apply Lemma \ref{completion} to the final sum, followed by Lemma \ref{Kbound}, to obtain 
\begin{eqnarray*}
\lefteqn{\sum_{x\in J(h_1,\ldots,h_k)}W_{\chi_0,\overline{q/q_0}h}(x;q_1h_1,\ldots,q_kh_k)}\\
&\ll_{k,\epsilon}& q_0^{2^{k-1}+\epsilon}\left(\frac{K}{\sqrt{q_0}}\left(q_0,\prod_{i=1}^k q_ih_i\right)^{\frac12}+\frac{1}{\sqrt{q_0}}\sum_{x\not\equiv 0\pmod {q_0}}\frac{1}{\|x/q_0\|}\left(q_0,x,\prod_{i=1}^k q_ih_i\right)^{\frac12}\right)\\
&\ll_{k,\epsilon}& q_0^{2^{k-1}+\epsilon}\left(\frac{K}{\sqrt{q_0}}\left(q_0,\prod_{i=1}^k h_i\right)^{\frac12}+\sqrt{q_0}\sum_{0<x\leq q_0/2}\frac{1}{x}\left(q_0,x,\prod_{i=1}^k h_i\right)^{\frac12}\right).\\
\end{eqnarray*}
The second claim in Lemma \ref{gcdlem} yields 
$$\sum_{0<x\leq q_0/2}\frac{1}{x}\left(q_0,x,\prod_{i=1}^k h_i\right)^{\frac12}\ll_\epsilon q_0^\epsilon$$
and therefore, recalling that $K=\frac{q}{N}\leq q_0$, 
\begin{eqnarray*}
\lefteqn{\sum_{x\in J(h_1,\ldots,h_k)}W_{\chi_0,\overline{q/q_0}h}(x;q_1h_1,\ldots,q_kh_k)}\\
&\ll_{k,\epsilon}& q_0^{2^{k-1}+\epsilon}\left(\frac{K}{\sqrt{q_0}}\left(q_0,\prod_{i=1}^k h_i\right)^{\frac12}+\sqrt{q_0}\right)\\
&\ll_{k,\epsilon}& q_0^{\frac{2^k+1}{2}+\epsilon}\left(q_0,\prod_{i=1}^k h_i\right)^{\frac12}\left(\frac{K}{q_0}+1\right)\\
&\ll_{k,\epsilon}& q_0^{\frac{2^k+1}{2}+\epsilon}\left(q_0,\prod_{i=1}^k h_i\right)^{\frac12}.\\
\end{eqnarray*}
Consequently, we deduce that 
\begin{eqnarray*}
\lefteqn{\left|\sum_{x\in J}\frac{1}{\sqrt q}W_{\chi,h}(x)e_q(-Mx) \right|^{2^k}}\\
&\ll_{\epsilon,k}& q^\epsilon\left(\sum_{j=1}^k K^{2^k-2^{k-j}}q_{k-j+1}^{2^{k-j}}\right.\\
&&\left.+K^{2^k-k-1}(q/q_0^{\frac12})\sum_{0<|h_1|\leq K/q_1}\ldots\sum_{0<|h_k|\leq K/q_k}\left(q_0,\prod_{i=1}^k h_i\right)^{\frac12}\right).\\
\end{eqnarray*}
Finally we use Lemma \ref{gcdlem} to estimate 
\begin{eqnarray*}
\sum_{0<|h_1|\leq K/q_1}\ldots\sum_{0<|h_k|\leq K/q_k}\left(q_0,\prod_{i=1}^k h_i\right)^{\frac12}&\leq&\sum_{0<|h|\leq \frac{K^k}{q_1\ldots q_k}}\tau_k(h)(h,q_0)^{\frac12}\\ 
&\ll_{k,\epsilon}&q^\epsilon\sum_{0<|h|\leq \frac{K^kq_0}{q}}(h,q_0)\\ 
&\ll_{k,\epsilon}&q^\epsilon\frac{K^kq_0}{q}\\
\end{eqnarray*}
from which we obtain 
\begin{eqnarray*}
\lefteqn{\left|\sum_{x\in J}\frac{1}{\sqrt q}W_{\chi,h}(x)e_q(-Mx) \right|^{2^k}}\\
&\ll_{\epsilon,k}& q^\epsilon\left(\sum_{j=1}^k K^{2^k-2^{k-j}}q_{k-j+1}^{2^{k-j}}+K^{2^k-1}q_0^{\frac12}\right).\\
\end{eqnarray*}
This completes the proof of Lemma \ref{BAkB} in the case that $\frac{q}{N}\leq q_0$.  

If $\frac{q}{N}>q_0$ we will deduce the result by combining some of the factors and then working with a smaller value of $k$.  Specifically we let $l$ be the smallest integer for which 
$$\prod_{i=0}^l q_i >\frac{q}{N},$$
which exists since $\frac{q}{N}\leq q$.  We may now apply the case of Lemma \ref{BAkB} which we have already proven with the factorisation 
$$q=\prod_{i=0}^{k-l}r_i$$
where 
$$r_0=\prod_{i=0}^l q_i$$
and
$$r_i=q_{i+l}\text{ for }i>l.$$
We obtain 
\begin{eqnarray*}
T&\ll_{\epsilon,k}& q^\epsilon\left(\sum_{j=1}^{k-l}N^{2^{-j}}q^{1/2-2^{-j}}r_{k-l-j+1}^{2^{-j}}+N^{2^{l-k}}q^{1/2-2^{l-k}}r_0^{1/2^{k-l+1}}\right)\\
&=& q^\epsilon\left(\sum_{j=1}^{k-l}N^{2^{-j}}q^{1/2-2^{-j}}q_{k-j+1}^{2^{-j}}+N^{2^{l-k}}q^{1/2-2^{l-k}}\left(\prod_{i=0}^l q_i\right)^{1/2^{k-l+1}}\right).\\
\end{eqnarray*}
By definition of $l$ we know that 
$$\prod_{i=0}^{l-1}q_i\leq \frac{q}{N}$$
and therefore 
$$N^{2^{l-k}}q^{1/2-2^{l-k}}\left(\prod_{i=0}^l q_i\right)^{1/2^{k-l+1}}
\leq N^{2^{l-k}}q^{1/2-2^{l-k}}\left(\frac{qq_l}{N}\right)^{1/2^{k-l+1}}
=N^{2^{l-k-1}}q^{1/2-2^{l-k-1}}q_l^{2^{l-k-1}}.$$
We conclude that 
\begin{eqnarray*}
T&\ll_{\epsilon,k}& q^\epsilon\left(\sum_{j=1}^{k-l}N^{2^{-j}}q^{1/2-2^{-j}}q_{k-j+1}^{2^{-j}}+N^{2^{l-k-1}}q^{1/2-2^{l-k-1}}q_l^{2^{l-k-1}}\right).\\
&=&q^\epsilon\sum_{j=1}^{k-l+1}N^{2^{-j}}q^{1/2-2^{-j}}q_{k-j+1}^{2^{-j}}\\
\end{eqnarray*}
This  is sharper than the required estimate and therefore completes the proof of Lemma \ref{BAkB}.

\section{Proof of Theorem \ref{Lthm}}

We modify  the proof given by Iwaniec and Kowalski in \cite[Theorem 12.9]{ik}, replacing the Burgess bound by Theorem \ref{charthm}.  We have 
$$L(\frac12,\chi)\ll \left|\sum_n \frac{\chi(n)}{\sqrt{n}}V(\frac{n}{\sqrt{q}})\right|,$$
where $V$ is a certain smooth function which satisfies 
$$V(y)\ll (1+y)^{-1}$$
and 
$$V'(y)\ll \frac{1}{y}(1+y)^{-1}.$$
We have 
$$\left(x^{-\frac12}V\left(\frac{x}{\sqrt{q}}\right)\right)' \ll x^{-\frac32}\left(1+\frac{x}{\sqrt{q}}\right)^{-1}$$
so we may apply summation by parts to deduce that 
$$L(\frac12,\chi)\ll \int_1^\infty\left|\sum_{n\leq x}\chi(n)\right|x^{-\frac32}\left(1+\frac{x}{\sqrt{q}}\right)^{-1}\,dx.$$
By Theorem \ref{charthm} we obtain 
$$\sum_{n\leq x}\chi(n)\ll_\epsilon x^{\frac{23}{41}}q^{\frac{11}{82}+O(\delta)+\epsilon}$$
and therefore 
$$L(\frac12,\chi)\ll_\epsilon q^{\frac{11}{82}+O(\delta)+\epsilon}\left(\int_1^{\sqrt{q}}x^{\frac{23}{41}-\frac32}\,dx+\sqrt{q}\int_{\sqrt{q}}^\infty x^{\frac{23}{41}-\frac52}\,dx\right)\ll q^{\frac{27}{164}+O(\delta)+\epsilon}.$$

\addcontentsline{toc}{section}{References} 
\bibliographystyle{plain}
\bibliography{../biblio}

\begin{thebibliography}{10}

\bibitem{blomermil}
V.~Blomer and D.~Mili\'cevi\'c.
\newblock The second moment of twisted modular $l$-functions.
\newblock arXiv:1404.7845.

\bibitem{burgess57}
D.~A. Burgess.
\newblock The distribution of quadratic residues and non-residues.
\newblock {\em Mathematika}, 4:106--112, 1957.

\bibitem{burgess62}
D.~A. Burgess.
\newblock On character sums and {$L$}-series.
\newblock {\em Proc. London Math. Soc. (3)}, 12:193--206, 1962.

\bibitem{burgess63}
D.~A. Burgess.
\newblock On character sums and {$L$}-series. {II}.
\newblock {\em Proc. London Math. Soc. (3)}, 13:524--536, 1963.

\bibitem{chang}
M.-C. Chang.
\newblock Short character sums for composite moduli.
\newblock {\em J. Anal. Math.}, 123:1--33, 2014.

\bibitem{deligne2}
P.~Deligne.
\newblock La conjecture de {W}eil. {II}.
\newblock {\em Inst. Hautes \'Etudes Sci. Publ. Math.}, (52):137--252, 1980.

\bibitem{fkm}
{\'E}.~Fouvry, E.~Kowalski, and P.~Michel.
\newblock A study in sums of products.
\newblock {\em Phil. Trans. R. Soc. A.}, 2015.
\newblock {to appear}.

\bibitem{gallagher}
P.~X. Gallagher.
\newblock Primes in progressions to prime-power modulus.
\newblock {\em Invent. Math.}, 16:191--201, 1972.

\bibitem{goldmakher}
L.~Goldmakher.
\newblock Character sums to smooth moduli are small.
\newblock {\em Canad. J. Math.}, 62(5):1099--1115, 2010.

\bibitem{gk}
S.~W. Graham and G.~Kolesnik.
\newblock {\em van der {C}orput's method of exponential sums}, volume 126 of
  {\em London Mathematical Society Lecture Note Series}.
\newblock Cambridge University Press, Cambridge, 1991.

\bibitem{grahamringrose}
S.~W. Graham and C.~J. Ringrose.
\newblock Lower bounds for least quadratic nonresidues.
\newblock In {\em Analytic number theory ({A}llerton {P}ark, {IL}, 1989)},
  volume~85 of {\em Progr. Math.}, pages 269--309. Birkh\"auser Boston, Boston,
  MA, 1990.

\bibitem{rhbhybrid}
D.~R. Heath-Brown.
\newblock Hybrid bounds for {D}irichlet {$L$}-functions.
\newblock {\em Invent. Math.}, 47(2):149--170, 1978.

\bibitem{rhbx32}
D.~R. Heath-Brown.
\newblock The largest prime factor of {$X^3+2$}.
\newblock {\em Proc. London Math. Soc. (3)}, 82(3):554--596, 2001.

\bibitem{mydivisor}
A.~J. Irving.
\newblock The divisor function in arithmetic progressions to smooth moduli.
\newblock {\em Int. Math. Res. Not. IMRN}, 2014.
\newblock {to appear}.

\bibitem{iwaniecdirichlet}
H.~Iwaniec.
\newblock On zeros of {D}irichlet's {$L$} series.
\newblock {\em Invent. Math.}, 23:97--104, 1974.

\bibitem{ik}
H.~Iwaniec and E.~Kowalski.
\newblock {\em Analytic number theory}, volume~53 of {\em American Mathematical
  Society Colloquium Publications}.
\newblock American Mathematical Society, Providence, RI, 2004.

\bibitem{milicevic}
D.~Mili\'cevi\'c.
\newblock Sub-{W}eyl subconvexity for {D}irichlet l-functions to prime power
  moduli.
\newblock {\em Compos. Math.}, 2014.
\newblock {to appear}.

\bibitem{postnikov}
A.~G. Postnikov.
\newblock On {D}irichlet {$L$}-series with the character modulus equal to the
  power of a prime number.
\newblock {\em J. Indian Math. Soc. (N.S.)}, 20:217--226, 1956.

\bibitem{weil1}
A.~Weil.
\newblock {\em Sur les courbes alg\'ebriques et les vari\'et\'es qui s'en
  d\'eduisent}.
\newblock Actualit\'es Sci. Ind., no. 1041 = Publ. Inst. Math. Univ. Strasbourg
  {\bf 7} (1945). Hermann et Cie., Paris, 1948.

\end{thebibliography}

\bigskip
\bigskip

Centre de recherches math\'ematiques,

Universit\'e de Montr\'eal,

Pavillon Andr\'e-Aisenstadt,

2920 Chemin de la tour, Room 5357,

Montr\'eal (Qu\'ebec) H3T 1J4

\bigskip 

{\tt alastair.j.irving@gmail.com}
\newpage


\section*{Supplementary material}

# APPENDIX: SUMS OF A. IRVING

ÉTIENNE FOUVRY, EMMANUEL KOWALSKI, AND PHILIPPE MICHEL

Let $p$ be a prime, $h \in \mathbf{F}_p^\times$ a fixed invertible element modulo $p$. Let $\chi$ be a fixed non-trivial multiplicative character modulo $p$.

We consider the function $W : \mathbf{F}_p \longrightarrow \mathbf{C}$ defined by
$$W(x) = \frac{1}{\sqrt{p}} \sum_{n \in \mathbf{F}_p} \chi(n)\overline{\chi(n+h)}e_p(nx).$$

In the notation of A. Irving's paper, we have $W(x) = p^{-1/2}W_{\chi,h}(x)$.

The goal is to study sums of products of values of $W$ of the type
$$S(y) = \sum_{x \in \mathbf{F}_p} \prod_{i=1}^{n_1} W(x + t_i) \prod_{j=1}^{n_2} \overline{W(x + s_j)} e_p(xy)$$

where the $t_i$'s and $s_j$'s are in $\mathbf{F}_p$ and $y \in \mathbf{F}_p$. The sum $K_{\chi,h}(h_1, \ldots, h_k, y)$ in Irving's paper is of this type, up to a factor $p^{-2^k/2}$, with $n_1 = n_2 = 2^{k-1}$ for some integer $k \geqslant 1$ and the $(t_i)$ (resp. $(s_j)$'s) are the sums
$$\sum_{i \in I} h_i$$
where $I$ ranges over $I \subset \{1, \ldots, k\}$ with $|I|$ even (resp. odd).

We will prove:

**Theorem 1.** *Assume $p > 5$. We have*
$$|S(y)| \ll \sqrt{p}$$
*where the implied constant depends only on $n_1$ and $n_2$ if any of the following conditions hold:*
  (1) *If $y + \frac{(n_2-n_1)h}{2} \neq 0$;*
  (2) *If $y + (n_2 - n_1)h/2 = 0$ and there is some $x \in \mathbf{F}_p$ such that the total multiplicity*
$$|\{i \mid t_i = x\}| + |\{j \mid s_j = x\}|$$
*is odd.*

The proof is based on the methods of algebraic geometry described in [1], based on the Riemann Hypothesis over finite fields of Deligne and additional ideas and computations of Katz [2]. We begin by a remark that allows us to remove the complex conjugates, and which is needed to apply cleanly the results of [1]. We define
$$W'(x) = W(x)e_p(\bar{2}hx).$$

---

1

**Lemma 2.** *We have*
$$S(y) = e_p(\alpha h) T\left(y + \frac{(n_2 - n_1)h}{2}\right)$$
*where*
$$\alpha = \frac{1}{2}\left(\sum_i t_i - \sum_j s_j\right)$$
*and*
$$T(y) = \sum_{x \in \mathbf{F}_p} \prod_{l=1}^{n_1+n_2} W'(x + u_l) e_p(xy)$$
*where $u_l = t_l$ for $1 \leqslant l \leqslant n_1$ and $u_l = s_{l-n_1}$ for $n_1 + 1 \leqslant l \leqslant n_1 + n_2$.*

*Proof.* This is a straightforward computation where the point is that $W'(x)$ is real for all $x \in \mathbf{F}_p$; this follows in turn from the relation
$$e_p(-hx)\overline{W(x)} = \frac{1}{\sqrt{p}} \sum_{n \in \mathbf{F}_p} \overline{\chi(n)} \chi(n+h) e_p(-(n+h)x)$$
$$= \frac{1}{\sqrt{p}} \sum_{m \in \mathbf{F}_p} \overline{\chi(-m-h)} \chi(-m) e_p(mx) = W(x).$$
□

Hence the result follows from the following result concerning sums like $T$:

**Proposition 3.** *Let $(u_1, \ldots, u_n)$ be distinct elements of $\mathbf{F}_p$, let $\nu_i \geqslant 1$. Let*
$$T(y) = \sum_{x \in \mathbf{F}_p} \prod_{i=1}^{n} W'(x + u_i)^{\nu_i} e_p(xy)$$
*for $y \in \mathbf{F}_p$. Then we have $T \ll p^{1/2}$, where the implied constant depends only on $n$ and the $\nu_i$, provided either that $y \neq 0$, or that $y = 0$ and some $\nu_i$ is odd.*

*Proof.* We will first apply [1, Th. 2.7] and then [1, Prop. 1.1] to the following data:
- the sheaf $\mathcal{G}$ is $\mathcal{L}_{\psi(-yX)}$, with $\psi$ the additive character corresponding to $e_p$,
- the family of sheaves is $(\mathcal{F}_i)_{1 \leqslant i \leqslant n}$ with
$$\mathcal{F}_i = [+u_i]^* \mathcal{F},$$
where $\mathcal{F} = \mathrm{FT}_\psi(\mathcal{H}) \otimes \mathcal{L}_{\psi(hX/2)}$, with $\mathcal{H}$ the sheaf-theoretic Fourier transform of the Kummer sheaf
$$\mathcal{L}_{\chi(X)\bar{\chi}(X+h)}.$$
- the open set $U$ is the complement of $\{-u_i\}$ in the affine line.

By definition of the Fourier transform, the trace function of $\mathcal{F}$ is the function $W'$ (one must check that the tensor product defining $\mathcal{H}$ is a middle-extension, but this is straightforward), and that of $\mathcal{F}_i$ is $W'(x + u_i)$, so the sum $T$ is exactly of the type controlled by [1, Prop. 1.1] in that case.

We now check that the family $(\mathcal{F}_i)$ is strictly $U$-generous, in the sense of [1, Def. 2.1], which is the hypothesis in [1, Th. 2.7]. This is true because:

2

(1) The $\mathcal{F}_i$ are geometrically irreducible middle-extensions of weight 0, because so is $\mathcal{F}$, in turn because it is a tensor product of an Artin-Schreier sheaf (of weight 0 and of rank 1) with the Fourier transform of the middle-extension sheaf $\mathcal{H}$ of weight 0 which is geometrically irreducible and of Fourier type (not being an Artin-Schreier sheaf); as noted above, the middle-extension property is checked directly.

(2) The geometric monodromy group of $\mathcal{F}$ is equal to $\mathrm{SL}_2$ for $p > 5$ (see below for the justification), so that the same holds for $\mathcal{F}_i$, and this group satisfies the second condition of loc. cit.,

(3) Any pairs of $\mathrm{SL}_2$ with their standard representations is Goursat-adapted (see [2, Example 1.8.1, p. 25]), so the third condition holds;

(4) For $i \neq j$, we have no geometric isomorphism
$$\mathcal{F}_i \simeq \mathcal{F}_j \otimes \mathcal{L}, \text{ or } \mathrm{D}(\mathcal{F}_i) \simeq \mathcal{F}_j \otimes \mathcal{L},$$
of sheaves on $U$, for any rank 1 sheaf $\mathcal{L}$ lisse on $U$ (see also below), where $\mathrm{D}(\mathcal{F}_i)$ is the dual of $\mathcal{F}_i$.

Leaving for a bit later the last checks indicated, we can finish the proof of the proposition first. From [1, Prop. 1.1], it is enough to show that the abstract "diagonal" classification of [1, Th. 2.7] implies that
$$H_c^2(U \times \bar{\mathbf{F}}_p, \bigotimes \mathcal{F}_i^{\otimes \nu_i} \otimes \mathrm{D}(\mathcal{G})) = 0$$
unless $y = 0$ and all $\nu_i$ are even. But [1, Th. 2.7] shows that this cohomology group vanishes *unless* we have a geometric isomorphism
$$\mathcal{G} \simeq \bigotimes \Lambda_i(\mathcal{F}_i)$$
where $\Lambda_i$ is an irreducible representation of $\mathrm{SL}_2$ contained in the $\nu_i$-th tensor of its standard representation (the covering $\pi$ in loc. cit. is trivial here because we are in a strictly $U$-generous situation). Assume we have such an isomorphism. Then, since the rank of $\mathcal{G}$ is 1, we see that $\Lambda_i$ is trivial for all $i$, and therefore $\mathcal{G}$ is also geometrically trivial, which means that $y = 0$. Then we see that $\nu_i$ must be even for the trivial representation to be contained in the $\nu_i$-th tensor power. (Note that this is the same argument as for sums of products of classical Kloosterman sums in [1, Cor. 3.2]).

We now finish checking the conditions (2) and (4) above. For (2), we first need to prove that the connected component of the identity $G^0$ of the geometric monodromy group $G$ of $\mathcal{F}$ is $\mathrm{SL}_2$. Since the tensor product with $\mathcal{L}_{\psi(hX/2)}$ does not alter this property, this is a special case of [2, Th. 7.9.4], because $\mathcal{H}$ is a tame pseudoreflection sheaf, ramified only at $\infty$ and at the points in $S = \{0, -h\} \subset \mathbf{A}^1$ (see also [2, Th. 7.9.6]). We then need to check that in fact $G$ is connected so that $G = G^0 = \mathrm{SL}_2$. This is because $\mathcal{F}$ is geometrically self-dual, as follows from the formula [2, Th. 7.3.8 (2)] for the dual of a Fourier transform (intuitively, this is because $W'$ is real-valued). Indeed, $G$ is semisimple and hence of the form $\mu_N \mathrm{SL}_2$ for some $N \geqslant 1$, where $\mu_N$ is the group of $N$-th roots of unity. This group is self-dual in the standard representation if and only if $N \leqslant 2$, but $\mu_2 \mathrm{SL}_2 = \mathrm{SL}_2$.

For (4), assume first that we have a geometric isomorphism
$$\mathcal{F}_i = [+u_i]^* \mathcal{F} \simeq [+u_j]^* \mathcal{F} \otimes \mathcal{L} = \mathcal{F}_j \otimes \mathcal{L} \tag{1}$$
(on $U$) for some rank 1 sheaf $\mathcal{L}$.

By [2, Cor. 7.4.6] (1), the sheaf $\mathcal{F}$ is tamely ramified at 0 with drop 1, and it is unramified on $\mathbf{G}_m$ by [2, Th. 7.9.4]. If $i \neq j$, then the left-hand sheaf in (1) is therefore unramified at $-u_j$, since $u_i \neq u_j$. On the other hand, $\mathcal{F}_j$ is ramified at $-u_j$ and has rank 2 and drop 1. Thus the tensor product with the rank 1 sheaf $\mathcal{L}$ is ramified at $-u_j$ (the dimension of the inertial invariants is at most 1), which is a contradiction. Thus an isomorphism as above is impossible unless $i = j$.

There remains to deal with the possibility of an isomorphism
$$\mathcal{F}_i = [+u_i]^* \mathcal{F} \simeq [+u_j]^* \operatorname{D}(\mathcal{F}) \otimes \mathcal{L} = \operatorname{D}(\mathcal{F}_j) \otimes \mathcal{L}.$$
But we have seen that $\operatorname{D}(\mathcal{F}) \simeq \mathcal{F}$, so this is also impossible if $i \neq j$. □

## References

[1] É. Fouvry, E. Kowalski and Ph. Michel: *A study in sums of products*, Phil. Trans. R. Soc. A, to appear; dx.doi.org/10.1098/rsta.2014.0309

[2] N.M. Katz: *Exponential sums and differential equations*, Annals of Math. Studies 124, Princeton Univ. Press (1990).

Université Paris Sud, Laboratoire de Mathématique, Campus d'Orsay, 91405 Orsay Cedex, France
*E-mail address*: etienne.fouvry@math.u-psud.fr

ETH Zürich – D-MATH, Rämistrasse 101, CH-8092 Zürich, Switzerland
*E-mail address*: kowalski@math.ethz.ch

EPFL/SB/IMB/TAN, Station 8, CH-1015 Lausanne, Switzerland
*E-mail address*: philippe.michel@epfl.ch



\end{document}